\documentclass[11pt,reqno]{amsart}

\usepackage[margin=1.2in]{geometry}
\usepackage[english]{babel}
\usepackage{amsmath,amssymb,amsfonts,amsthm,mathtools}
\usepackage{dsfont}
\usepackage{hyperref}
\usepackage{tikz-cd}

\hypersetup{
  colorlinks = true,
  urlcolor = blue,
  linkcolor = blue,
  citecolor = black
}

\numberwithin{equation}{section}
\theoremstyle{plain}
\newtheorem{thmctr}{}[section]
\newtheorem{lemma}[thmctr]{Lemma}
\newtheorem{theorem}[thmctr]{Theorem}

\newtheorem{corollary}[thmctr]{Corollary}
\theoremstyle{definition}

\theoremstyle{remark}
\newtheorem{remark}[thmctr]{Remark}

\newcommand{\kk}{\Bbbk}
\newcommand{\unit}{\mathds{1}}

\newcommand{\cA}{\mathcal{A}}
\newcommand{\B}{\mathcal{B}}
\newcommand{\C}{\mathcal{C}}
\newcommand{\D}{\mathcal{D}}
\newcommand{\E}{\mathcal{E}}
\newcommand{\F}{\mathcal{F}}
\newcommand{\Z}{\mathcal{Z}}
\newcommand{\T}{\mathcal{T}}
\newcommand{\cS}{\mathcal{S}}

\newcommand{\loc}{\mathrm{loc}}
\newcommand{\Hom}{\mathrm{Hom}}
\newcommand{\ra}{\mathrm{ra}}
\newcommand{\im}{\mathrm{Im}}
\newcommand{\FP}{\mathrm{d}}
\newcommand{\Vect}{\mathrm{Vect}}
\newcommand{\id}{\mathrm{id}}

\newcommand{\oB}{\overline{\B}}

\newcommand{\bF}{\mathbf{F}}

\newcommand{\U}{\mathrm{U}}

\title[Transparent subalgebras and local modules]{Transparent Subalgebras and Local Module Categories}
\author{Harshit Yadav and Kenichi Shimizu}
\date{\today}

\begin{document}
\begin{abstract}
Let $A$ be a commutative simple algebra in a braided finite tensor category
$\B$. We identify the largest transparent subalgebra of $A$ as the algebra 
induced by a central lift of the free-module functor.
This identification gives formulas for the Frobenius--Perron dimension and
the M\"uger center of the category of local $A$-modules. These formulas give
criteria for nondegeneracy, symmetry, and modularity, together with sharp
bounds on $\FP_{\B}(A)$. We also realize the M\"uger center of $\B$ as a
category of local modules over an adjoint algebra. Finally, we prove a
relative-center factorization and deduce that taking the category of local 
modules preserves the relative Witt class.
\end{abstract}

\maketitle


\section{Introduction}

Throughout, we work over an algebraically closed field $\kk$. 
Let $A$ be a commutative algebra in a braided tensor category $\B$. The
category $\B_A^{\loc}$ of local $A$-modules consists of the modules whose
$A$-action is compatible with the double braiding. It inherits a braiding
from $\B$. A basic structural problem is to determine the size of
$\B_A^{\loc}$ and the degeneracy of its braiding directly from $\B$ and $A$ 
\cite{pareigis1995braiding,kirillov2002q,davydov2013witt,laugwitz2022constructing,shimizu2026commutative}.

Suppose now that $\B$ is a braided finite tensor category and that $A$ is
simple.
The M\"uger center $\Z_2(\B)$ consists of the objects whose double braiding
with every object of $\B$ is the identity. We call a subalgebra of $A$
\emph{transparent} if it belongs to $\Z_2(\B)$, and write
$A':=A\cap\Z_2(\B)$ for the largest transparent subalgebra of $A$. Our main
result identifies $A'$ intrinsically as the algebra induced by the
inverse-braiding central lift of the free-module functor. Consequently, $A'$ controls both the Frobenius--Perron dimension and the M\"uger
center of $\B_A^{\loc}$.

\begin{theorem}[Main theorem]\label{thm:main}
Let $\B$ be a braided finite tensor category and let $A$ be a commutative
simple algebra in $\B$. Let $\oB$ denote $\B$ equipped with the reversed
braiding, and let
\[
  F_A^-:\oB\longrightarrow\Z(\B_A)
\]
be the inverse-braiding central lift of the free-module functor. 
\begin{enumerate}
  \item There is an isomorphism
  \[
    A' \cong (F_A^-)^{\ra}(\unit_{\Z(\B_A)})
  \]
  of subalgebras of $A$. Moreover, the corestriction of $F_A^-$ induces a
  braided equivalence
  \[
    \im(F_A^-) \simeq \overline{\B_{A'}}.
  \]
  \item The Frobenius--Perron dimension of the local-module category is
  \[
    \FP(\B_A^{\loc})
    =
    \frac{\FP(\B)\FP_{\B}(A')}{\FP_{\B}(A)^2}.
  \]
  \item There are braided equivalences
  \[
    \Z_2(\B_A^{\loc})
    \simeq
    \Z_2(\B_{A'})
    \simeq
    \bigl(\Z_2(\B)\bigr)_{A'}.
  \]
\end{enumerate}
\end{theorem}

After introducing the preliminaries in \S\ref{sec:preliminaries}, this result is proved in \S\ref{sec:main-theorem}.  
Part (1) is the main input for the rest of the paper. To explain the proof, set
$\D:=\im(F_A^-)$. A centralizer theorem identifies $\B_A^{\loc}$ with the
M\"uger centralizer of $\D$ inside $\Z(\B_A)$. Part~(1) identifies
$\D$ with $\overline{\B_{A'}}$. Since $\Z(\B_A)$ is nondegenerate, the
centralizer dimension formula gives Part~(2), while the double-centralizer
theorem gives Part~(3).

The final section \S\ref{sec:applications} contains three applications of this result. First,
the formulas show explicitly how the transparent part of $A$ affects the
local-module category. For example, $\B_A^{\loc}$ is nondegenerate if and
only if
\[
  \FP_{\B}(A')=\FP\bigl(\Z_2(\B)\bigr).
\]
When $\B$ is nondegenerate, $A'\cong\unit$, so the correction term
$\FP_{\B}(A')$ disappears. At the opposite extreme, if $\B$ is symmetric,
then $A'=A$ and every $A$-module is local. We also obtain a factorization of
$\Z(\B_A)$ when $\B_A^{\loc}$ is nondegenerate, sharp bounds on
$\FP_{\B}(A)$, and a criterion for $\B_A^{\loc}$ to be modular. 

Second, set $\mathcal{B}^{\mathrm{env}} =\oB\boxtimes\B$ and consider the tensor-product
functor
\[
  F:\B^{\mathrm{env}}\longrightarrow\B,
  \qquad
  X\boxtimes Y\longmapsto X\otimes Y.
\]
Its right adjoint determines the adjoint algebra $\bF:=F^\ra(\unit_\B)$ in
$\B^{\mathrm{env}}$. Theorem~\ref{thm:adjoint-algebra-local-modules}
identifies
\[
  (\B^{\mathrm{env}})_\bF^{\loc}
  \simeq
  \Z_2(\B).
\]
Thus, the M\"uger center itself is realized as a local-module category. To prove this, we show that, for braided finite tensor
categories $\C$ and $\D$, $\Z_2(\C\boxtimes\D) \simeq \Z_2(\C)\boxtimes\Z_2(\D)$; this may be of independent interest.

Finally, let $\E$ be a finite symmetric tensor category and suppose that
$\B$ is equipped with a symmetric tensor equivalence
$\E\simeq\Z_2(\B)$. If $A\cap\E\cong\unit$, we prove the relative-center
factorization
\[
  \Z(\B_A,\E)
  \simeq
  \B\boxtimes_{\E}\overline{\B_A^{\loc}}.
\]
It follows that $\B$ and $\B_A^{\loc}$ are Witt equivalent over $\E$.

\subsection*{Acknowledgements}
HY was partly supported by a start-up grant from the University of Alberta and an NSERC Discovery Grant. KS was supported by JSPS KAKENHI Grant Number JP24K06676. The authors used ChatGPT while developing and drafting parts of this paper to explore possible formulations and proof strategies. All mathematical statements and proofs were independently checked by the authors, who take sole responsibility for the contents of the paper.


\section{Preliminaries}\label{sec:preliminaries}


\subsection{Induced algebras and monoidal adjunctions}
\label{subsec:induced-transparent-algebra}

Let $\C$, $\D$ and $\E$ be monoidal categories. 
We first recall how algebra maps behave under a composite monoidal
adjunction. If a strong monoidal functor $T:\C\to\D$ has a right adjoint,
we write
\[
  \zeta_T:\Hom_{\C}(X,T^{\ra}(Y))
  \longrightarrow
  \Hom_{\D}(T(X),Y)
\]
for the adjunction bijection. This bijection restricts to algebra maps when
$X$ and $Y$ are algebras.

\begin{lemma}\label{lem:alg-map-factorization}
Let $F:\C\to\D$ and $G:\D\to\E$ be strong monoidal functors with right
adjoints, and set $H=G\circ F$. Define
\[
  P:=H^{\ra}(\unit_{\E}),
  \qquad
  Q:=F^{\ra}(\unit_{\D}).
\]
The monoidal unit isomorphism $G(\unit_{\D})\cong\unit_{\E}$ induces an
algebra map $j:Q\to P$. A map $i:X\to P$ factors through $j$ if and only if
there is a map $\varphi:F(X)\to\unit_{\D}$ such that
\[
  \zeta_H(i)=G(\varphi),
\]
where we use the monoidal unit isomorphism for $G$. The same statement holds
for algebra maps.
\end{lemma}

\begin{proof}
The adjunction bijections for $F$, $G$, and $H$ form a commutative square
\[
\begin{tikzcd}[column sep=large,row sep=large]
\operatorname{Hom}_{\C}(X,Q)
  \arrow[r,"\zeta_F"]
  \arrow[d,"j\circ(-)"']
&
\operatorname{Hom}_{\D}(F(X),\unit_{\D})
  \arrow[d,"G(-)"]
\\
\operatorname{Hom}_{\C}(X,P)
  \arrow[r,"\zeta_H"']
&
\operatorname{Hom}_{\E}(H(X),\unit_{\E}).
\end{tikzcd}
\]
Since the horizontal maps are bijections, $i$ belongs to the image of the
left vertical map if and only if $\zeta_H(i)$ belongs to the image of the
right vertical map. Monoidal adjunctions restrict these bijections to algebra
maps, which proves the last assertion.
\end{proof}


\subsection{The inverse-braiding central lift}
Suppose that $\B$ is a braided monoidal category with coequalizers and the tensor product preserves them. Let $A$ be a commutative algebra in $\B$. We denote by $\B_A$ the monoidal category
of right $A$-modules. Its tensor product is the relative tensor product over
$A$, and its tensor unit is the regular module $A$. If $a^M:M\otimes A\to M$
is the action of $A$ on $M$, then $M$ is \emph{local} if
\[
  a^M=a^M\circ c_{A,M}\circ c_{M,A}.
\]
The full subcategory of local modules is denoted by $\B_A^{\loc}$. It is
braided by the braiding induced from $\B$
\cite[Section~2]{pareigis1995braiding}.

We call a commutative algebra $K$ in $\B$ \emph{transparent} if
$K\in\Z_2(\B)$. In this case every $K$-module is local, so the braiding of
$\B$ induces a braiding on $\B_K$.

An algebra map is always assumed to preserve the multiplication and the unit.
A \emph{subalgebra} of $A$ is an algebra equipped with a monic algebra map to
$A$. An \emph{ideal} of $A$ is a subobject preserved by left and right
multiplication, and $A$ is \emph{simple} if its only ideals are $0$ and $A$.

For $X\in\B$ and $M\in\B_A$, define
\[
\begin{aligned}
\sigma^-_{X,M}:(X\otimes A)\otimes_A M
&\cong X\otimes M
\xrightarrow{c_{M,X}^{-1}} M\otimes X
\cong M\otimes_A(X\otimes A).
\end{aligned}
\]
This is one of the two standard central structures on the free-module
functor; see
\cite[Exercises~8.8.3--8.8.5 and Proposition~8.8.10]{etingof2016tensor}.
It defines a braided tensor functor
\[
  F_A^-:\oB\longrightarrow\Z(\B_A).
\]
where $\oB$ is the same monoidal category with the reversed braiding $\overline{c}_{X,Y} = c_{Y,X}^{-1}$.

We next make explicit the locality calculation used in the proof of the Main
Theorem.

\begin{lemma}\label{lem:minus-half-braiding-locality}
Let $L\subseteq A$ be a subalgebra, with inclusion $i_L:L\hookrightarrow A$,
and set
\[
  q_L:=m_A\circ(i_L\otimes\id_A):F_A(L)=L\otimes A\longrightarrow A
\]
be the algebra map corresponding to $i_L$ under the free-module adjunction.
For $(M,a^M)\in\B_A$, let
\[
  b^M:=a^M\circ(\id_M\otimes i_L):M\otimes L\longrightarrow M
\]
be the restricted right $L$-action. The half-braiding compatibility condition
for $q_L$ at $M$ holds if and only if the restricted module $M$ is local over
$L$. Consequently, $q_L:F_A^-(L)\to A$ is a morphism in $\Z(\B_A)$ if and
only if every $A$-module is local after restriction to $L$.
\end{lemma}

\begin{proof}
The condition that $q_L$ preserve the half-braiding must hold for every
$M\in\B_A$. Under the canonical identifications
\[
  F_A^-(L)\otimes_A M\cong L\otimes M,
  \qquad
  M\otimes_A F_A^-(L)\cong M\otimes L,
\]
this condition is the equality
\[
  b^M\circ c_{M,L}^{-1}
  =
  b^M\circ c_{L,M}
\]
of maps $L\otimes M\to M$. Precomposing with $c_{M,L}$ gives
\[
  b^M=b^M\circ c_{L,M}\circ c_{M,L},
\]
which is precisely the locality condition for the restricted right
$L$-module.
\end{proof}


\subsection{Finite tensor categories}

We use the standard terminology
for finite tensor categories and tensor functors \cite{etingof2016tensor}. For a tensor functor $F:\C\to\D$, we write
$F^{\ra}$ for its right adjoint and $\im(F)$ for the tensor subcategory of
$\D$ consisting of subquotients of objects in the essential image of $F$.
The functor is called \emph{surjective} if $\im(F) = \D$. We write
$\boxtimes$ for the Deligne product. A tensor subcategory $\D$ of $\C$
is a full subcategory that contains $\unit_\C$ and is closed under tensor
products, direct sums, duals, and subquotients.

We denote the Frobenius--Perron dimension of an object $X\in\C$ by
$\FP_{\C}(X)$ and that of $\C$ by $\FP(\C)$; see
\cite[Section~4]{etingof2016tensor}.

Let $\B$ be a braided finite tensor category with braiding $c$. 
If $\cA$ is a full tensor subcategory of $\B$, we denote its M\"uger
centralizer in $\B$ by $\Z_2(\cA\subseteq\B)$. 
Thus $\Z_2(\cA\subseteq\B)$ is the full tensor subcategory consisting of
the objects $X\in\B$ such that $c_{Y,X}\circ c_{X,Y}=\id_{X\otimes Y}$ 
for every $Y\in\cA$. The \emph{M\"uger center} of a braided finite tensor
category $\B$ is
\[
  \Z_2(\B):=\Z_2(\B\subseteq\B).
\]
The category $\B$ is called \emph{nondegenerate} if $\Z_2(\B)\simeq\Vect$.


\subsection{Local modules}
Let $\B$ be a braided finite tensor category and $A$ a commutative simple algebra in $\B$. We use the following facts about commutative simple algebras.

\begin{lemma}\label{lem:simple-commutative-module-categories}
\begin{enumerate}
  \item The category $\B_A$ is a finite tensor category and
  $\B_A^{\loc}$ is a braided finite tensor category.
  \item Every subalgebra of $A$ is simple.
\end{enumerate}
\end{lemma}

\begin{proof}
By \cite[Theorem~7.1]{coulembier2025simple}, the algebra $A$ is exact.
Moreover, $\Hom_{\B}(\unit,A)\cong\kk$ by
\cite[Lemma~5.9]{shimizu2026commutative}. Part~(1) now follows from
\cite[Theorem~5.5]{shimizu2026commutative}.

Let $K\subseteq A$ be a subalgebra. By
\cite[Proposition~5.6]{ostrik2026non}, the algebra $K$ is exact. Since the
monomorphism $K\to A$ preserves the unit, we have
$\Hom_{\B}(\unit,K)\cong\kk$. The commutative exact algebra $K$ is
indecomposable by \cite[Lemma~5.8]{shimizu2026commutative}, and hence simple
by
\cite[Theorem~7.1]{coulembier2025simple}.
\end{proof}

\begin{lemma}\label{lem:transparent-muger-center}
Let $K$ be a simple transparent algebra in $\B$. Then
$\Z_2(\B_K)=\bigl(\Z_2(\B)\bigr)_K$ as full braided subcategories of
$\B_K$.
\end{lemma}

\begin{proof}
Every object of $\bigl(\Z_2(\B)\bigr)_K$ is transparent in $\B_K$, since
its underlying object has trivial double braiding with every object of
$\B$. Hence
$\bigl(\Z_2(\B)\bigr)_K\subseteq\Z_2(\B_K)$.

Conversely, let $M\in\Z_2(\B_K)$. For $X\in\B$, consider the free module
$F_K(X)=X\otimes K$. Under the canonical identifications
\[
  M\otimes_K F_K(X)\cong M\otimes X,
  \qquad
  F_K(X)\otimes_K M\cong X\otimes M,
\]
the double braiding of $M$ with $F_K(X)$ is induced by the double braiding
of the underlying object of $M$ with $X$ in $\B$. Since $M$ is transparent
in $\B_K$, this double braiding is the identity. This holds for every
$X\in\B$, so the underlying object of $M$ belongs to $\Z_2(\B)$. Thus
$M\in\bigl(\Z_2(\B)\bigr)_K$.
\end{proof}

The free-module functor $F_A:\B\to\B_A$, $X\mapsto X\otimes A$, is
surjective. The standard dimension formula for a surjective tensor functor
gives
\begin{equation}\label{eq:module-category-dimension}
  \FP(\B_A)=\frac{\FP(\B)}{\FP_{\B}(A)};
\end{equation}
see \cite[Lemma~6.2.4]{etingof2016tensor}.


Let $i:\C\hookrightarrow\D$ be the inclusion of a tensor subcategory of a
finite tensor category. For $X\in\D$, write
\[
  X\cap\C:=i i^{\ra}(X).
\]
The counit identifies $X\cap\C$ with the
largest subobject of $X$ that belongs to $\C$. In particular, if $A$ is an
algebra in a braided finite tensor category $\B$, then
\[
  A':=A\cap\Z_2(\B)
\]
is an algebra in $\Z_2(\B)$, and its counit map $A'\hookrightarrow A$ is a
monic algebra map.


\section{Proof of the Main Theorem}
\label{sec:main-theorem}

Let $\B$ be a braided finite tensor category and $A$ a commutative simple algebra in $\B$. 
Set $\D:=\im(F_A^-)$ and consider its M\"uger centralizer 
inside $\Z(\B_A)$. Then
\cite[Theorem~5.13]{shimizu2026commutative} gives a braided equivalence\footnote{We use \cite[Theorem~5.13]{shimizu2026commutative} with the correction that the functor $G$ in its proof there is replaced by $\overline{\mathcal B}\to\mathcal Z(\mathcal B)$ and the functor $G_A$ by the induced functor $\overline{\B}\to\Z(\B_A)$. With this replacement, the proof applies verbatim and yields \eqref{eq:centralizer-of-lift}.}
\begin{equation}\label{eq:centralizer-of-lift}
  \B_A^{\loc}
  \simeq
  \Z_2(\D\subseteq\Z(\B_A)).
\end{equation}


\subsection{Proof of Theorem~\ref{thm:main}(1)}
Let $\U:\Z(\B_A)\to\B_A$ be the forgetful functor and let $\U^{\ra}$ be its
right adjoint. Set $K:=(F_A^-)^\ra(\unit)$.
Since $\U\circ F_A^-=F_A$ as tensor functors, the adjunctions
give an algebra map
\[
  k:K\longrightarrow (F_A^-)^{\ra}(\U^{\ra}(A))
  \cong F_A^{\ra}(A)
  \cong A.
\]
The morphism $\unit_{\Z(\B_A)}\to\U^{\ra}(A)$ corresponding to $\id_A$ is
nonzero. It is monic because the tensor unit of $\Z(\B_A)$ is simple. The
right adjoint $(F_A^-)^{\ra}$ is left exact, so $k$ is also monic.

Corestricting $F_A^-$ gives a surjective braided tensor functor  
$f:\oB\longrightarrow\D$. 
For every $X\in\oB$, we have
\[
\begin{aligned}
  \Hom_{\D}(f(X),\unit_{\D})
  &=\Hom_{\Z(\B_A)}(F_A^-(X),\unit_{\Z(\B_A)})
  \\
  &=\Hom_{\oB}(X,K).
\end{aligned}
\]
Thus $f^{\ra}(\unit_{\D})\cong K$. By the comparison theorem
\cite[Proposition~6.6]{shimizu2026commutative}, there is a tensor equivalence
$\D\simeq(\oB)_K$ under which $f$ becomes the free-module functor. This
free-module functor is braided, so
\cite[Lemma~5.10]{shimizu2026commutative} gives
$K\in\Z_2(\oB)=\Z_2(\B)$. The map $k$ therefore identifies $K$ with a
transparent subalgebra of $A$. Hence $K\subseteq A'$.

Let $j:A'\hookrightarrow A$ be the canonical algebra map. Under
the free-module adjunction, it gives
\[
  q_{A'}=m_A\circ(j\otimes\id_A):F_A(A')\longrightarrow A.
\]
Since $A'\in\Z_2(\B)$, the double braiding of $A'$ with every object of $\B$
is the identity. The restricted $A'$-action on every $M\in\B_A$ is therefore
local. Lemma~\ref{lem:minus-half-braiding-locality} shows that $q_{A'}$ is a
morphism in $\Z(\B_A)$ from $F_A^-(A')$ to the tensor unit.

Apply Lemma~\ref{lem:alg-map-factorization} to
\[
  \B\xrightarrow{F_A^-}\Z(\B_A)\xrightarrow{\U}\B_A,
\]
where $F_A^-$ is regarded as a tensor functor from the common underlying
monoidal category of $\B$ and $\oB$. The fact that $q_{A'}$ is the image
under $\U$ of a morphism in $\Z(\B_A)$ implies that $j$ factors through
$k$. Hence $A'\subseteq K$, and therefore $K\cong A'$ as subalgebras of
$A$.

It remains to identify the braiding on the image. Since
$A'\in\Z_2(\B)$, we have $c_{A',M}=c_{M,A'}^{-1}$ for every $M\in\B$.
Thus the monoidal category $(\oB)_{A'}$ has the same
underlying monoidal structure as $\B_{A'}$, while its braiding is the
reversed braiding. Transporting the braiding of $\D$ along the comparison
equivalence makes the surjective free-module functor braided. This determines
the braiding on $(\oB)_{A'}$, so the comparison equivalence is braided and 
$\D\simeq\overline{\B_{A'}}$. \qed


\subsection{Proof of Theorem~\ref{thm:main}(2)}
By \eqref{eq:centralizer-of-lift}, we have
$\B_A^{\loc}\simeq\Z_2(\D\subseteq\Z(\B_A))$. The Drinfeld center
$\Z(\B_A)$ is factorizable
\cite[Proposition~8.6.3]{etingof2016tensor}, and hence it is nondegenerate
\cite[Theorem~4.2]{shimizu2019non}.
The centralizer dimension formula
\cite[Theorem~4.9]{shimizu2019non} therefore gives
\[
  \FP(\B_A^{\loc})\FP(\D)
  =
  \FP(\Z(\B_A)).
\]
Moreover, $\FP(\Z(\B_A))=\FP(\B_A)^2$ 
by \cite[Theorem~7.16.6]{etingof2016tensor}. Using
\eqref{eq:module-category-dimension}, we obtain
\[
  \FP(\B_A^{\loc})
  =
  \frac{\FP(\B)^2}{\FP_{\B}(A)^2\FP(\D)}.
\]

Part~(1) gives $\D\simeq\overline{\B_{A'}}$. Applying
\eqref{eq:module-category-dimension} to $A'$ yields
\[
  \FP(\D)
  =
  \frac{\FP(\B)}{\FP_{\B}(A')}.
\]
Substitution gives
\[
  \FP(\B_A^{\loc})
  =
  \frac{\FP(\B)\FP_{\B}(A')}{\FP_{\B}(A)^2},
\]
as required. This finishes the proof. \qed

\begin{remark}
If $\B$ is nondegenerate, then $A'\cong\unit$, and the formula becomes
\[
  \FP(\B_A^{\loc})
  =
  \frac{\FP(\B)}{\FP_{\B}(A)^2}.
\]
If $\B$ is symmetric, then $A'=A$ and every $A$-module is local. The formula
then becomes
\[
  \FP(\B_A^{\loc})
  =
  \frac{\FP(\B)}{\FP_{\B}(A)}
  =
  \FP(\B_A).
\]
\end{remark}


\subsection{Proof of Theorem~\ref{thm:main}(3)}
Set
\[
  \E:=\Z_2(\D\subseteq\Z(\B_A)).
\]
The equivalence \eqref{eq:centralizer-of-lift} identifies
$\B_A^{\loc}$ with $\E$ inside the nondegenerate category $\Z(\B_A)$.
The double-centralizer theorem \cite[Theorem~4.9]{shimizu2019non} gives 
$\Z_2(\E\subseteq\Z(\B_A))=\D$. 
Hence
\[
  \Z_2(\E) = \E\cap\Z_2(\E\subseteq\Z(\B_A)) = \E\cap\D = \Z_2(\D).
\]
By Part~(1), $\D\simeq\overline{\B_{A'}}$, and therefore
$\Z_2(\B_A^{\loc}) \simeq \Z_2(\B_{A'})$. 
Since $A'$ is transparent, Lemma~\ref{lem:transparent-muger-center} gives
$\Z_2(\B_{A'})=\bigl(\Z_2(\B)\bigr)_{A'}$, 
which proves Part~(3). \qed

\begin{remark}
Lemma~\ref{lem:simple-commutative-module-categories} shows that $A'$ is simple
in $\B$. Since $\Z_2(\B)$ is a full tensor subcategory and
$A'\in\Z_2(\B)$, it is also simple as an algebra in $\Z_2(\B)$. The standard
dimension formula \cite[Lemma~5.15]{shimizu2026commutative} now gives
\begin{equation}\label{eq:muger-center-local-module-fpdim}
  \FP\bigl(\Z_2(\B_A^{\loc})\bigr)
  =
  \frac{\FP\bigl(\Z_2(\B)\bigr)}{\FP_{\B}(A')}.
\end{equation}
\end{remark}


\section{Applications}
\label{sec:applications}
As in the previous section, $\B$ is a braided finite tensor category and $A$ is a simple commutative algebra in $\B$.


\subsection{Consequences for local modules}\label{sec:consequences}
We continue to write $A'=A\cap\Z_2(\B)$.

\begin{corollary}\label{cor:local-module-nondegeneracy}
The following conditions are equivalent:
\begin{enumerate}
  \item $\B_A^{\loc}$ is nondegenerate.
  \item $\B_{A'}$ is nondegenerate.
  \item $\FP_{\B}(A')=\FP\bigl(\Z_2(\B)\bigr)$.
  \item There exist a finite affine group scheme $G$ and a symmetric tensor
  equivalence
  \[
    \Phi:\Z_2(\B)\xrightarrow{\ \sim\ }\operatorname{Rep}(G)
  \]
  such that $\Phi(A')\cong\mathcal O(G)$ as commutative algebras, where
  $\mathcal O(G)$ is the regular algebra in $\operatorname{Rep}(G)$.
\end{enumerate}
\end{corollary}
\begin{proof}
The equivalence of (1) and (2) follows from
Theorem~\ref{thm:main}(3). Equation
\eqref{eq:muger-center-local-module-fpdim} shows that (1) is equivalent to
(3).

Under condition (3), we have
$\bigl(\Z_2(\B)\bigr)_{A'}\simeq\Vect$. Since $\Z_2(\B)$ is symmetric,
the free-module functor
\[
  F_{A'}:\Z_2(\B)\longrightarrow
  \bigl(\Z_2(\B)\bigr)_{A'}\simeq\Vect
\]
is a symmetric fiber functor. Tannaka reconstruction
\cite[Theorem~5.3.12]{etingof2016tensor} gives a finite-dimensional Hopf
algebra $H$ whose representation category, together with its forgetful
functor, is equivalent to $(\Z_2(\B),F_{A'})$. Since $F_{A'}$ is symmetric,
$H$ is cocommutative. Thus $H^*$ is the coordinate algebra $\mathcal O(G)$
of a finite affine group scheme $G$, and we obtain a symmetric tensor
equivalence
\[
  \Phi:\Z_2(\B)\xrightarrow{\ \sim\ }\operatorname{Rep}(G).
\]
The algebra induced by $F_{A'}$ is
$(F_{A'})^{\ra}(\unit)\cong A'$, while the algebra induced by the forgetful
fiber functor on $\operatorname{Rep}(G)$ is the regular algebra
$\mathcal O(G)$. Hence $\Phi(A')\cong\mathcal O(G)$ as commutative algebras,
which proves (4).

Conversely, suppose that (4) holds. The equivalence $\Phi$ induces a tensor
equivalence
\[
  \bigl(\Z_2(\B)\bigr)_{A'}\simeq
  \operatorname{Rep}(G)_{\mathcal O(G)}\simeq\Vect.
\]
The dimension formula then gives (3).
\end{proof}

If $A'\cong\unit$, Theorem~\ref{thm:main}(3) gives
$\Z_2(\B_A^{\loc})\simeq\Z_2(\B)$. Thus, the condition
$A'\cong\unit$ alone does not imply that
$\B_A^{\loc}$ is nondegenerate. In this case, it is nondegenerate if and
only if $\B$ is nondegenerate.

The criterion is more general than the nondegeneracy of $\B$. 
For example, suppose that $\B = \C \boxtimes \mathrm{Rep}(G)$ for some nondegenerate braided finite tensor category $\C$ and a finite group scheme $G$. We consider the algebra $A = \unit \boxtimes \mathcal{O}(G)$. Then, since $A'=A$ and $\mathrm{Rep}(G)_{\mathcal{O}(G)} \simeq \Vect$, we have $\B_{A'} = \B_A \simeq \C$. Therefore $\B_A^{\loc}$ is nondegenerate, even though $\B$ need not be nondegenerate.

Now we go back to the setting of Corollary \ref{cor:local-module-nondegeneracy}.

\begin{corollary}\label{cor:local-module-factorization}
Suppose that the equivalent conditions in
Corollary~\ref{cor:local-module-nondegeneracy} hold. Then there is a braided
equivalence
\[
  \Z(\B_A)
  \simeq
  \overline{\B_{A'}}\boxtimes\B_A^{\loc}.
\]
\end{corollary}

\begin{proof}
Theorem~\ref{thm:main}(1) identifies $\D$ with
$\overline{\B_{A'}}$, so the hypothesis says that $\D$ is nondegenerate.
By \eqref{eq:centralizer-of-lift},
\[
  \Z_2(\D\subseteq\Z(\B_A))\simeq\B_A^{\loc}.
\]
The result follows from
\cite[Theorem~4.17]{laugwitz2022relative}.
\end{proof}

\begin{corollary}\label{cor:local-module-numerical-consequences}
The following statements hold.
\begin{enumerate}
  \item We have
  \[
    \FP_{\B}(A)^2
    \leq
    \FP(\B)\FP_{\B}(A').
  \]
  Equality holds if and only if $\B_A^{\loc}\simeq\Vect$.
  \item If $\B$ is integral, then
  \[
    \frac{\FP(\B)\FP_{\B}(A')}{\FP_{\B}(A)^2}
  \]
  is an integer.
\end{enumerate}
\end{corollary}

\begin{proof}
Part~(1) follows from Theorem~\ref{thm:main}(2), since a finite tensor
category has Frobenius--Perron dimension at least one, with equality precisely
for $\Vect$. If $\B$ is integral, then $\B_A^{\loc}$ is integral by
\cite[Lemma~5.16]{shimizu2026commutative}. Its Frobenius--Perron dimension is
therefore an integer, which proves Part~(2).
\end{proof}

\begin{corollary}\label{cor:local-module-sharp-dimension-bound}
We have $\FP_{\B}(A)^2\FP\bigl(\Z_2(\B)\bigr) \leq \FP(\B)\FP_{\B}(A')^2$.
Equality holds if and only if $\B_A^{\loc}$ is symmetric.
In particular,
\[
  \FP_{\B}(A)
  \leq
  \sqrt{\FP(\B)\FP\bigl(\Z_2(\B)\bigr)}.
\]
\end{corollary}

\begin{proof}
Theorem~\ref{thm:main}(2) and
\eqref{eq:muger-center-local-module-fpdim} give
\[
  \frac{
    \FP(\B_A^{\loc})
  }{
    \FP\bigl(\Z_2(\B_A^{\loc})\bigr)
  }
  =
  \frac{
    \FP(\B)\FP_{\B}(A')^2
  }{
    \FP_{\B}(A)^2\FP\bigl(\Z_2(\B)\bigr)
  }.
\]
Since the M\"uger center is a tensor subcategory of $\B_A^{\loc}$, the
left-hand side is at least one. Equality holds precisely when
\[
  \Z_2(\B_A^{\loc})=\B_A^{\loc},
\]
which means that $\B_A^{\loc}$ is symmetric. Finally,
$\FP_{\B}(A')\leq\FP\bigl(\Z_2(\B)\bigr)$ gives the last bound.
\end{proof}

A \emph{modular tensor category} is a nondegenerate ribbon finite tensor
category. We regard $A^*$ as a right $A$-module via the action dual to the
left regular action on $A$.
\begin{corollary}\label{cor:local-module-modularity}
Suppose that $\B$ is ribbon, $\theta_A=\id_A$, and
$A^*\in\Z_2(\B_A^{\loc})$. If $\B_{A'}$ is nondegenerate, then
$\B_A^{\loc}$ is a modular tensor category. 
In particular, the conclusion holds if $A$ is symmetric Frobenius and
$\B_{A'}$ is nondegenerate.
\end{corollary}
\begin{proof}
By \cite[Theorem~5.21(b)]{shimizu2026commutative}, the assumptions on the
twist and on $A^*$ imply that $\B_A^{\loc}$ is ribbon. It is nondegenerate by
Corollary~\ref{cor:local-module-nondegeneracy}, and hence it is a modular
tensor category. If $A$ is symmetric Frobenius, then $\theta_A=\id_A$ by
\cite[Lemma~2.6]{shimizu2026commutative}. The Frobenius property also gives
$A^*\cong A$ as right $A$-modules. Thus $A^*$ lies in
$\Z_2(\B_A^\loc)$, so the first part applies.
\end{proof}


\subsection{Adjoint algebra and the M\"uger center}
We first prove the following result on the M\"uger center of a Deligne tensor product of braided finite tensor categories. This result is required later, but may be of independent interest.

\begin{theorem}\label{thm:muger-center-deligne-product}
Let $\C$ and $\D$ be braided finite tensor categories, and
equip $\C\boxtimes\D$ with the product braiding. Then
\[
\Z_2(\C\boxtimes\D)
\simeq
\Z_2(\C)\boxtimes\Z_2(\D).
\]
In particular, $\C\boxtimes\D$ is nondegenerate if and only if
both $\C$ and $\D$ are nondegenerate.
\end{theorem}

\begin{proof}
Write
\[
\E=\C\boxtimes\D
\qquad\text{and}\qquad
\cS=\C\boxtimes\unit.
\]
We use the centralizer dimension formula
\cite[Theorem~4.9]{shimizu2019non}:
\[
\FP(\cA)
\FP\bigl(\Z_2(\cA\subseteq\mathcal X)\bigr)
=
\FP(\mathcal X)
\FP\bigl(\cA\cap\Z_2(\mathcal X)\bigr)
\]
for every tensor full subcategory $\cA$ of a braided finite tensor
category $\mathcal X$.

The product braiding gives
\[
\cS\cap\Z_2(\E)
=
\Z_2(\C)\boxtimes\unit
\]
and an inclusion
\[
\Z_2(\C)\boxtimes\D
\subseteq
\Z_2(\cS\subseteq\E).
\]
Applying the dimension formula to $\cS\subseteq\E$, we obtain
\[
\begin{aligned}
\FP\bigl(\Z_2(\cS\subseteq\E)\bigr)
=
\frac{
 \FP(\E)
 \FP\bigl(\cS\cap\Z_2(\E)\bigr)
}{
 \FP(\cS)
} 
=
\FP\bigl(\Z_2(\C)\bigr)
\FP(\D).
\end{aligned}
\]
Thus, by \cite[Proposition~6.3.3]{etingof2016tensor}, the preceding inclusion is an equivalence:
\[
  \Z_2(\cS\subseteq\E)\simeq\Z_2(\C)\boxtimes\D.
\]

Set $\F = \Z_2(\C)\boxtimes\D$ and $\T=\unit\boxtimes\D$. 
Since $\cS$ and $\T$ tensor-generate $\E$, an object of
$\E$ is transparent precisely when it lies in
$\F=\Z_2(\cS\subseteq\E)$ and centralizes $\T$. Then
\[
\Z_2(\E) = \Z_2(\T\subseteq\F).
\]
Moreover, the product braiding gives
\[
\T\cap\Z_2(\F)
=
\unit\boxtimes\Z_2(\D)
\]
and
\[
\Z_2(\C)\boxtimes\Z_2(\D)
\subseteq
\Z_2(\T\subseteq\F).
\]
Applying the dimension formula inside $\F$, we find
\begin{equation*}
\FP\bigl(\Z_2(\T\subseteq\F)\bigr)
=
\frac{
  \FP(\F)\FP\bigl(\T\cap\Z_2(\F)\bigr)
}{
  \FP(\T)
}
= \FP\bigl(\Z_2(\C)\bigr) \FP\bigl(\Z_2(\D)\bigr).
\end{equation*}
The last quantity is the Frobenius--Perron dimension of
$\Z_2(\C)\boxtimes\Z_2(\D)$. Therefore the
preceding inclusion is an equality, and
\[
\Z_2(\E)
=
\Z_2(\C)\boxtimes\Z_2(\D).
\]

The final assertion follows because a braided finite tensor category is
nondegenerate precisely when its M\"uger center is equivalent to $\Vect$.
\end{proof}

Let $\B$ be a braided finite tensor category, and set
$\B^{\mathrm{env}}:=\oB\boxtimes\B$. Equip $\B^{\mathrm{env}}$ with the
product braiding. The tensor-product functor
\[
  F:\B^{\mathrm{env}}\longrightarrow\B,
  \qquad
  F(X\boxtimes Y)=X\otimes Y,
\]
has a standard central structure. Set $\bF:=F^\ra(\unit)$; this is the
\emph{adjoint algebra} associated with $F$. By
\cite[Proposition~6.6]{shimizu2026commutative}, the algebra $\bF$ is
indecomposable, exact, and commutative. It is therefore simple by
\cite[Theorem~7.1]{coulembier2025simple}.

The functor $F$ is surjective because $F(\unit\boxtimes X)\cong X$ for every
$X\in\B$. The comparison equivalence
\cite[Proposition~2.3]{etingof2004analogue} gives $(\B^{\mathrm{env}})_\bF\simeq\B$.

\begin{theorem}\label{thm:adjoint-algebra-local-modules}
There is a braided equivalence
\[
  (\B^{\mathrm{env}})_\bF^{\loc}\simeq\Z_2(\B).
\]
\end{theorem}

\begin{proof}
Set $\bF':=\bF\cap\Z_2(\B^{\mathrm{env}})$. By
Theorem~\ref{thm:muger-center-deligne-product},
\[
  \Z_2(\B^{\mathrm{env}})
  \simeq
  \overline{\Z_2(\B)}\boxtimes\Z_2(\B).
\]
Since $\Z_2(\B)$ is symmetric, the restriction of $F$ corestricts to a
surjective symmetric tensor functor
\[
  F_0:\overline{\Z_2(\B)}\boxtimes\Z_2(\B)
  \longrightarrow\Z_2(\B),
  \qquad
  F_0(X\boxtimes Y)=X\otimes Y.
\]
The functor $F_0$ is surjective because
$F_0(\unit\boxtimes X)\cong X$ for every $X\in\Z_2(\B)$. Define
\[
  \mathbf{B}:=F_0^\ra(\unit).
\]
By \cite[Proposition~6.6]{shimizu2026commutative}, the algebra $\mathbf B$
is indecomposable, exact, and commutative in
$\bigl(\Z_2(\B)\bigr)^{\mathrm{env}}$. It is therefore simple by
\cite[Theorem~7.1]{coulembier2025simple}.

Let $i:\Z_2(\B^{\mathrm{env}})\hookrightarrow\B^{\mathrm{env}}$ and
$j:\Z_2(\B)\hookrightarrow\B$ be the inclusions. The equality
$j\circ F_0=F\circ i$ gives an isomorphism of right adjoints
\[
  F_0^{\ra}\circ j^{\ra}\cong i^{\ra}\circ F^\ra.
\]
The tensor unit of $\B$ belongs to $\Z_2(\B)$, so
$j^{\ra}(\unit_{\B})\cong\unit_{\Z_2(\B)}$. Evaluating the preceding
isomorphism at $\unit_{\B}$ gives
\[
  \mathbf{B} = F_0^{\ra}(\unit) \cong i^{\ra} F^\ra(\unit) = \bF'.
\]

Under the comparison equivalence, an object $X\in\B$ corresponds to
$(\unit\boxtimes X)\otimes\bF$. If $X\in\Z_2(\B)$, then
$\unit\boxtimes X\in\Z_2(\B^{\mathrm{env}})$, so the corresponding free
module is local. Hence the comparison equivalence restricts to a fully
faithful braided tensor functor
\[
  T:\Z_2(\B)\longrightarrow(\B^{\mathrm{env}})_\bF^\loc.
\]

It remains to compare Frobenius--Perron dimensions. We have
\[
\begin{aligned}
\FP\bigl((\B^{\mathrm{env}})_\bF^\loc\bigr)
&=
\frac{
  \FP(\B^{\mathrm{env}})\FP_{\B^{\mathrm{env}}}(\bF')
}{
  \FP_{\B^{\mathrm{env}}}(\bF)^2
}
\\
&=
\FP(\B^{\mathrm{env}})
\frac{\FP(\B)^2}{\FP(\B^{\mathrm{env}})^2}
\frac{
  \FP\bigl((\Z_2(\B))^{\mathrm{env}}\bigr)
}{
  \FP(\Z_2(\B))
}
\\
&=
\FP(\Z_2(\B)).
\end{aligned}
\]
The first equality uses Theorem~\ref{thm:main}(2). The second uses
\cite[Proposition~6.6(d)]{shimizu2026commutative} for $F$ and $F_0$.
Thus the source and target of $T$ have the same Frobenius--Perron dimension.
Since $T$ is fully faithful, it is a braided tensor equivalence.
\end{proof}


\subsection{Relative centers and Witt equivalence}
\label{subsec:relative-centers-witt}

Let $\E$ be a finite symmetric tensor category. Following
\cite[Definition~4.1]{davydov2013structure} and
\cite[\S5.3]{decoppetstroinski2026pre}, an
\emph{$\E$-nondegenerate} braided finite tensor category is a braided finite
tensor category $\B$ equipped with a symmetric tensor equivalence
\[
  \E\xrightarrow{\sim}\Z_2(\B).
\]
Via the specified equivalence, we regard $\E$ as the full tensor subcategory
$\Z_2(\B)\subseteq\B$. Since $\E$ is symmetric, the reverse of an
$\E$-nondegenerate braided finite tensor category is again
$\E$-nondegenerate. For braided finite tensor categories equipped with braided
tensor embeddings of $\E$ into their M\"uger centers, we write
$\boxtimes_{\E}$ for their relative Deligne product.

Let $\B$ be a braided finite tensor category. 
For $X\in\B$ and $M\in\B_A$, define
\[
\begin{aligned}
  \sigma^+_{X,M}:(X\otimes A)\otimes_A M
  &\cong X\otimes M
  \xrightarrow{c_{X,M}} M\otimes X
  \cong M\otimes_A(X\otimes A).
\end{aligned}
\]
This gives the positive central lift $P:=F_A^+:\B\longrightarrow\Z(\B_A)$. 
The following theorem is a nonsemisimple analogue of
\cite[Corollary~4.6]{davydov2013structure}.

\begin{theorem}\label{thm:relative-center-local-modules}
Suppose that $\B$ is $\E$-nondegenerate and $A$ is
a commutative simple algebra in $\B$ such that $A\cap\E\cong\unit$. 
Then, the free-module functor restricts to a symmetric tensor equivalence
$\E\xrightarrow{\sim}\Z_2(\B_A^{\loc})$. In particular,
$\B_A^{\loc}$ is $\E$-nondegenerate. Moreover, the positive
central lift $P$ is fully faithful, and there is a braided equivalence over
$\E$
\[
  \Z_2(P(\E)\subset \Z(\B_A))
  \simeq
  \B\boxtimes_{\E}\overline{\B_A^{\loc}}.
\]
\end{theorem}

We prove the theorem in three steps. 

\begin{lemma}\label{lem:relative-positive-lift}
The functor $P$ is fully faithful. 
\end{lemma}

\begin{proof}
Set $K:=P^{\ra}(\unit_{\Z(\B_A)})$. Since the composite of $P$ with the
forgetful functor is $F_A$, the construction in the proof of
Theorem~\ref{thm:main}(1) gives a monic algebra map $K\hookrightarrow A$.
The comparison argument in the same proof shows that $K\in\Z_2(\B)$,
because $P$ is braided. Hence $K\subseteq A\cap\E\cong\unit$, and therefore
$K\cong\unit$. The comparison theorem
\cite[Proposition~6.6]{shimizu2026commutative} identifies the corestriction
of $P$ with the free-module functor $\B\to\B_K$. Thus $P$ is fully faithful,
and so is its restriction to $\E$.
\end{proof}

Locality gives a second central lift. For $M\in\B_A^\loc$ and $N\in\B_A$, the map
$c_{N,M}^{-1}:M\otimes N\to N\otimes M$ descends to the relative tensor
products over $A$. These maps define a braided tensor functor $Q:\overline{\B_A^\loc}\longrightarrow\Z(\B_A)$. 
For $E\in\E$, the free module $F_A(E)$ is local. Hence the free-module
functor restricts to a braided tensor functor
\[
  G:\E\longrightarrow\B_A^\loc.
\]

\begin{lemma}\label{lem:relative-mutual-centralizers}
Set $\cS:=\im(P)$ and $\T:=\im(Q)$. Then $\cS$ and $\T$ are mutual
M\"uger centralizers in $\Z(\B_A)$.
Moreover, there is a monoidal isomorphism $P|_{\E}\cong Q\circ G$, 
and $G$ induces a symmetric tensor equivalence
$\E\simeq\Z_2(\B_A^\loc)$. 
In particular, $\T\cap\cS=\im(P |_\E)$.
\end{lemma}

\begin{proof}
We apply the centralizer theorem used in \eqref{eq:centralizer-of-lift}, namely
\cite[Theorem~5.13]{shimizu2026commutative}, to the positive central lift
$P$. With the above convention for half-braidings, the central functor from
the local-module category in that theorem is
$Q:\overline{\B_A^\loc}\to\Z(\B_A)$. Thus $Q$ is fully faithful, and
\begin{equation}\label{eq:relative-positive-centralizer}
  \T = \Z_2\bigl(\cS\subseteq\Z(\B_A)\bigr).
\end{equation}
The double-centralizer theorem \cite[Theorem~4.9]{shimizu2019non} then gives
\begin{equation}\label{eq:relative-negative-centralizer}
  \cS
  =
  \Z_2\bigl(\T\subseteq\Z(\B_A)\bigr).
\end{equation}

For $E\in\E$, the positive half-braiding on $P(E)$ agrees with the inverse
half-braiding on $Q(G(E))$, since $E$ is transparent. These identifications
give a monoidal isomorphism
\begin{equation}\label{eq:relative-central-lifts-agree}
  P|_{\E}\cong Q\circ G.
\end{equation}
Since $P$ and $Q$ are fully faithful, so is $G$. Moreover,
$Q(G(E))\cong P(E)$ belongs to $\T\cap\cS$, and $\cS$ centralizes $\T$.
Thus $G$ takes values in $\Z_2(\B_A^\loc)$. Finally,
Theorem~\ref{thm:main}(3), with $A'=A\cap\E\cong\unit$, gives
$\FP(\Z_2(\B_A^\loc))=\FP(\E)$. It follows from
\cite[Proposition~6.3.3]{etingof2016tensor} that
$G:\E\to\Z_2(\B_A^\loc)$ is a symmetric tensor equivalence.
Together with \eqref{eq:relative-negative-centralizer} and
\eqref{eq:relative-central-lifts-agree}, this gives
\begin{equation}\label{eq:relative-centralizer-intersection}
  \T\cap\cS = \Z_2(\T) =  Q\bigl(\Z_2(\B_A^\loc)\bigr) = P(\E).
\end{equation}
\end{proof}

\begin{lemma}\label{lem:relative-center-factorization}
The functor $\B\boxtimes\overline{\B_A^\loc}\to\Z(\B_A)$ defined by
$X\boxtimes M\mapsto P(X)\otimes Q(M)$ descends to a braided equivalence
$\B\boxtimes_{\E}\overline{\B_A^\loc}\simeq\Z_2(P(\E)\subset \Z(\B_A))$ over $\E$.
\end{lemma}

\begin{proof}
The isomorphism \eqref{eq:relative-central-lifts-agree} supplies an
$\E$-balancing. Explicitly, for $E\in\E$, it gives coherent isomorphisms
\[
  P(X\otimes E)\otimes Q(M)
  \cong
  P(X)\otimes Q\bigl(G(E)\otimes_A M\bigr).
\]
Let $\Phi:\B\boxtimes_{\E}\overline{\B_A^\loc}\to\Z(\B_A)$ be the resulting braided
tensor functor, and let $\mathcal K:=\cS\vee\T$ be its image. By
\eqref{eq:relative-positive-centralizer} and
\eqref{eq:relative-negative-centralizer}, together with
\eqref{eq:relative-centralizer-intersection},
\[
\begin{aligned}
  \Z_2\bigl(\mathcal K\subseteq\Z(\B_A)\bigr)
  &=
  \Z_2\bigl(\cS\subseteq\Z(\B_A)\bigr)
  \cap
  \Z_2\bigl(\T\subseteq\Z(\B_A)\bigr)
  \\
  &=
  \T\cap\cS
  =
  P(\E).
\end{aligned}
\]
Applying the double-centralizer theorem once more yields
\[
  \mathcal K
  =
  \Z_2\bigl(P(\E)\subseteq\Z(\B_A)\bigr).
\]
Thus $\Phi$ corestricts to a surjective braided tensor functor $\Phi: \B\boxtimes_{\E}\overline{\B_A^\loc} \longrightarrow \Z_2(P(\E) \subset \Z(\B_A))$.

It remains to compare dimensions. A similar argument as in the proof of 
\cite[Theorem~3.2.1]{decoppet2025invertibility}, together with the
standard dimension formula, gives the first equality below.
Theorem~\ref{thm:main}(2) and \eqref{eq:module-category-dimension} give the
second, while the center and centralizer dimension formulas used in the proof
of Theorem~\ref{thm:main}(2) give the third:
\[
  \FP\bigl(\B\boxtimes_{\E}\overline{\B_A^\loc}\bigr)
  =\frac{\FP(\B)\FP(\B_A^\loc)}{\FP(\E)}
  =\frac{\FP(\B_A)^2}{\FP(\E)}
  =\FP\bigl(\Z_2(P(\E)\subset \Z(\B_A))\bigr).
\]
The surjective tensor functor $\Phi$ is therefore an equivalence by
\cite[Proposition~6.3.4]{etingof2016tensor}. For every $E\in\E$, the image of $E$ in the source is represented by $E\boxtimes_{\E}\unit_{\B_A^\loc}$, and $\Phi$ sends this object to
$P(E)$. Hence $\Phi$ is an equivalence over $\E$.
\end{proof}

\begin{proof}[Proof of Theorem~\ref{thm:relative-center-local-modules}]
Lemma~\ref{lem:relative-positive-lift} proves that $P$, and hence its
restriction to $\E$, is fully faithful. The relative-center equivalence is
Lemma~\ref{lem:relative-center-factorization}. Finally,
Lemma~\ref{lem:relative-mutual-centralizers} identifies the M\"uger center of
$\B_A^{\loc}$ with $\E$, and hence proves the last assertion.
\end{proof}

We now record the consequence for relative Witt equivalence. 
An $\E$-enriched finite tensor category is a finite tensor category
$\C$ equipped with a braided tensor functor
$\iota_{\C}:\E\to\Z(\C)$. Following
\cite[\S5.3]{decoppetstroinski2026pre}, we call the enrichment
\emph{faithfully flat} if $\iota_{\C}$ is fully faithful. In this case, set
\[
  \Z(\C,\E)
  :=
  \Z_2\bigl(\iota_{\C}(\E)\subseteq\Z(\C)\bigr).
\] 

Via the canonical braided embedding
$\Z_2(\mathcal X)\hookrightarrow\Z(\mathcal X)$, every
$\E$-nondegenerate braided finite tensor category $\mathcal X$ is a
faithfully flat $\E$-enriched finite tensor category. Moreover,
Lemma~\ref{lem:relative-positive-lift} shows that
$P|_{\E}:\E\to\Z(\B_A)$ makes $\B_A$ a faithfully flat
$\E$-enriched finite tensor category. Its relative Drinfeld center is
\[
  \Z(\B_A,\E)
  =
  \Z_2\bigl(P(\E)\subseteq\Z(\B_A)\bigr).
\]

Following
\cite[Definition~5.9]{decoppetstroinski2026pre}, which generalizes
\cite[Definition~5.1]{davydov2013structure}, let $\B_1$ and $\B_2$ be
$\E$-nondegenerate braided finite tensor categories. We say
that they are \emph{Witt equivalent over $\E$} if there exist faithfully flat
$\E$-enriched finite tensor categories $\mathcal A_1$ and $\mathcal A_2$ and
a braided equivalence over $\E$
\[
  \B_1\boxtimes_{\E}\Z(\mathcal A_1,\E)
  \simeq
  \B_2\boxtimes_{\E}\Z(\mathcal A_2,\E).
\]
In this case, we write $[\B_1]_{\E}=[\B_2]_{\E}$.

\begin{corollary}\label{cor:relative-witt-local-modules}
Let $\B$ be an $\E$-nondegenerate braided finite tensor category and $A$ a commutative simple algebra in $\B$. If $A\cap\E\cong\unit$, then
$[\B]_{\E}=[\B_A^{\loc}]_{\E}$.
\end{corollary}

\begin{proof}
Applying Theorem~\ref{thm:relative-center-local-modules} to $(\B,A)$ shows
that $\B_A^{\loc}$ is $\E$-nondegenerate. We may therefore apply the theorem
again to $(\B_A^{\loc},\unit)$. The two resulting equivalences give
\[
  \B\boxtimes_{\E}\Z(\B_A^\loc,\E) \simeq
  \B\boxtimes_{\E}\B_A^\loc\boxtimes_{\E}\overline{\B_A^\loc}
  \simeq
  \B_A^\loc \boxtimes_{\E}\B\boxtimes_{\E}\overline{\B_A^\loc}
  \simeq
  \B_A^\loc \boxtimes_{\E}\Z(\B_A,\E).
\]
Thus, the assertion follows from the definition, with witnesses
$\mathcal A_1=\B_A^\loc$ and $\mathcal A_2=\B_A$. 
\end{proof}
This result is the finite-tensor category analogue of
\cite[Proposition~5.3]{davydov2013structure}.

\bibliographystyle{alpha}
\bibliography{references}

\end{document}